\documentclass[a4paper,12pt,intlimits,oneside]{amsart}
\usepackage{enumerate}
\usepackage{amsfonts,amsmath}
\usepackage{latexsym,amssymb}
\newcommand{\comment}[1]{}

\newcounter{rea}
\newcounter{rek}
\newcounter{res}
\begin{document}
\allowdisplaybreaks
\title[Multilinear Hardy-Ces\`{a}ro Operator and Commutator]{Multilinear Hardy-Ces\`{a}ro Operator and Commutator on the product of Morrey-Herz spaces}     

\author{Nguyen Minh CHUONG}    
\address{Institute of mathematics, Vietnamese  Academy of Science and Technology,  Hanoi, Vietnam.} 
\email{{\tt nmchuong@math.ac.vn}}
\author{Nguyen Thi HONG}    
\address{Hanoi Metropolitan University, 98 Duong Quang Ham, Hanoi, Vietnam.} 
\email{ nthong@hnmu.vn}
\author{$\text{Ha Duy HUNG}^\star$} 
\address{High School for Gifted Students,
Hanoi National University of Education, 136 Xuan Thuy, Hanoi, Vietnam} 
\email{{\tt hunghaduy@gmail.com}}
\thanks{This paper is supported by Vietnam National Foundation for Science and Technology Development (NAFOSTED) under grant number 101.02 - 2014.51.}
\thanks{$^\star$ Corresponding author.}
 

\begin{abstract} We obtain sufficient and necessary conditions on weight functions $s_1(t),\ldots,s_m(t)$ and $\psi(t)$ so that 
the weighted multilinear Hardy-Ces\`{a}ro operator 
\[(f_1,\ldots,f_m)\mapsto \int_{[0,1]^n}\left(\prod_{k=1}^nf_k\left(s_k(t) x\right)\right)\psi(t)dt
\]
is bounded from $\dot{K}^{\alpha_1, p_1}_{q_1}(\omega_1)\times \cdots  \times\dot{K}^{\alpha_m, p_m}_{q_m}(\omega_m)$ to $\dot{K}^{\alpha, p}_{q}(\omega)$ and from $M\dot{K}^{\alpha_1, \lambda_1}_{p_1,q_1}(\omega_1)\times \cdots  \times M\dot{K}^{\alpha_m, \lambda_m}_{p_m,q_m}(\omega_m)$ to $M\dot{K}^{\alpha, \lambda}_{p,q}(\omega)$. The sharp bounds are also obtained and these results hold for both cases $0<p<1$ and $1\leq p<\infty$. We give a sufficient condition so that if symbols $b_1,\ldots,b_m$ are Lipschitz, then the commutator of the weighted Hardy-Ces\`{a}ro operator
\[
(f_1,\ldots,f_m)\mapsto\int_{[0,1]^n}\left(\prod\limits_{k=1}^mf_k\left(s_k(t)x\right)\right)\left(\prod_{k=1}^m\left(b_k(x)-b_k\left(s_k(t)x\right)\right)\right)\psi(t)dt\]
is bounded from $M\dot{K}^{\alpha_1, \lambda_1}_{p_1, q_1}(\omega_1)\times \cdots \times M\dot{K}^{\alpha_m, \lambda_m}_{p_m, q_m}(\omega_m)$ to $M\dot{K}^{\alpha^\prime, \lambda}_{p, q}(\omega)$ for both cases $0<p<1$ and $1\leq p<\infty$. By these we extend and strengthen previous results deu to Tang, Xue, and Zhou \cite{TXZ}.
\end{abstract}

\maketitle
\newtheorem{theorem}{Theorem}[section]
\newtheorem{lemma}{Lemma}[section]
\newtheorem{proposition}{Proposition}[section]
\newtheorem{remark}{Remark}[section]
\newtheorem{corollary}{Corollary}[section]
\newtheorem{definition}{Definition}[section]
\newtheorem{example}{Example}[section]
\numberwithin{equation}{section}
\newtheorem{Theorem}{Theorem}[section]
\newtheorem{Lemma}{Lemma}[section]
\newtheorem{Proposition}{Proposition}[section]
\newtheorem{Remark}{Remark}[section]
\newtheorem{Corollary}{Corollary}[section]
\newtheorem{Definition}{Definition}[section]
\newtheorem{Example}{Example}[section]

\section{Introduction}

Let $\psi:(0,1]\to\mathbb R_{\ge0}$ then the weighted Hardy-Littlewood average operator $U_\psi$ is defined as the following
\begin{equation}\label{sec1eq1}
U_\psi f(x)=\int\limits_0^1f(tx)\psi(t)dt,\qquad x\in\mathbb R^d.
\end{equation}
In 2001, J. Xiao \cite{Xiao} obtained an interesting result that $U_\psi$ is bounded on $L^p(\mathbb R^d)$, with $1<p<\infty$, if and only if 
\begin{equation} \label{sec1eq2}
\int\limits_0^1t^{-d/p}\psi(t)dt<\infty.
\end{equation}
Moreover, the corresponding operator norm is
\begin{equation}\label{sec1eq3}
\|U_\psi\|_{L^p}=\int\limits_0^1t^{-d/p}\psi(t)dt.
\end{equation}
In case $\psi\equiv1$, $U_\psi$ reduces to the classical Hardy operator $S$, where $$Sf(x)=\frac1x\int_0^xf(t)dt.$$
Thus Xiao's result reduces the classical Hardy's integral inequality: if $f\ge0$, $1<p<\infty$ then
\[
\int_0^\infty\left(\frac1x\int_0^xf(t)dt\right)^pdx\leq \left(\frac p{p-1}\right)^p\int_0^\infty f(x)^pdx,
\]
and the constant $\left(\dfrac p{p-1}\right)^p$ is the best possible.
\vskip12pt
Xiao's results has been followed by a vast amount of research geared towards understanding the weighted Hardy-Littlewood average operators $U_\psi$, see for examples \cite{CH,FLL,FL,FL1,FL2,GFM,LF,TXZ} and the references therein. These includes the work of Chuong and Hung in \cite{CH}, where they introduced a class of integral transforms as follows.
\begin{definition}\label{def1} Let $\psi:[0,1]\to[0,\infty)$, $s:[0,1]\to\mathbb R$  be measurable functions. The weighted Hardy-Ces\`{a}ro operator $U_{\psi,s}$, associated to the parameter curve $s(x,t):=s(t)x$, is defined by
\begin{equation}\label{sec1eq4}
U_{\psi,s}f(x)=\int_0^1f\left(s(t) x\right)\psi(t)dt,
\end{equation}
for all measurable complex valued functions $f$ on $\mathbb R^d$.
\end{definition}
It turns out that such operators are still keeping almost all nice properties as the weighted Hardy-Littlewood average operators. With certain conditions on functions $s$ and $\omega$, the authors \cite{CH} proved $U_{\psi,s}$ is bounded on weighted Lebesgue spaces and weighted $BMO$ spaces. The corresponding operator norms are worked out too. The authors also give a necessary condition on the weight function $\psi$, for the boundedness of the commutators of operator $U_{\psi,s}$ on weighted Lebesgue spaces and BMO spaces, with homogeneous weights.
\vskip 12pt
Recently, Hung and Ky \cite{HK} introduced the weighted multilinear Hardy-Ces\`{a}ro operator which was defined as following
\begin{definition}\label{def2} Let $m, n \in \mathbb N, \psi:[0,1]^n \to[0,\infty)$, $s_1,\ldots,s_m :[0,1]^n \to\mathbb R$  be measurable functions. Given $f_1,\ldots, f_m: \mathbb R^d \to \mathbb C$ be measurable functions. 
 The weighted multilinear Hardy-Ces\`{a}ro operator $U^{m,n}_{\psi,\vec{s}}$, is defined by
\begin{equation}\label{sec1eq5}
U^{m,n}_{\psi,\vec{s}}\left(f_1,\ldots, f_m\right)(x)=\int_{[0,1]^n}\left(\prod_{k=1}^nf_k\left(s_k(t) x\right)\right)\psi(t)dt,
\end{equation}
where $\vec{s}=\left(s_1,\ldots,s_m\right).$
\end{definition}
When $s_k(t)\equiv t$, and $m=n$, $U^{m,n}_{\psi,\vec{s}}$ turns to $\mathcal H^m_\psi$ which was introduced by Fu etc. \cite{FGL}. In \cite{HK}, the authors obtain the sharp bounds of $U^{m,n}_{\psi,\vec{s}}$ on the product of Lebesgue spaces and central Morrey spaces. More details, under certain conditions, they proved that $U^{m,n}_{\psi,\overrightarrow{s}}$ is bounded from $L^{p_1}_{\omega_1}\left(\mathbb R^d\right)\times\cdots\times L^{p_m}_{\omega_m}\left(\mathbb R^d\right)$ to $L^{p}_{\omega}\left(\mathbb R^d\right)$ if and only if 
\begin{equation}\label{maineq1}
{\mathcal A}=\int_{[0,1]^n}\left(\prod_{i=1}^m|s_i(t)|^{-\frac{d+\gamma_i}{p_i}}\right)\psi(t)dt<\infty.
\end{equation}
Furthermore, 
\begin{equation}\label{sec1eq7}
\left\|U^{m,n}_{\psi,\overrightarrow{s}}\right\|_{L^{p_1}_{\omega_1}\left(\mathbb R^d\right)\times\cdots\times L^{p_m}_{\omega_m}\left(\mathbb R^d\right)\to L^{p}_{\omega}\left(\mathbb R^d\right)}={\mathcal A}.
\end{equation}

They also proved sufficient and necessary conditions of the weighted functions so that the commutators of $U^{m,n}_{\psi,\vec{s}}$ (with symbols in central BMO spaces) are bounded on the product of central Morrey spaces. Their results extends known results obtained by Fu, Gong, Lu and Yuan in \cite{FGL} and by Chuong, Hung in \cite{CH}.
\vskip12pt
It is well known that Herz and Morrey-Herz spaces are natural generalisations of Lebesgue spaces with power weights (see definitions in Section 2). So it is natural to study the boundedness and bounds of $U^{m,n}_{\psi,\vec{s}}$ on these functional spaces. Such problems for the weighted Hardy-Littlewood average operator $U_\psi$ are studied in \cite{FL,LF}. Results for the boundedness and bounds of $H^m_\psi$ on the product of Morrey- Herz spaces, was recently obtained by Gong, Fu and Ma in \cite{GFM}. In this paper, we obtain necessary and sufficient conditions for the weighted boundedness of the Hardy-Ces\`{a}ro operators for the product of Herz and Morrey-Herz spaces. In each cases, the estimates for operator norms are worked out. 

\vskip12pt 
On the other hand, recently Tang, Xue, and Zhou \cite{TXZ} find a sufficient condition on weights so that $U_\psi$ is bounded on Morrey-Herz spaces. In this paper we make use of another approach which allows us to obtain a sufficient condition in terms of a finite integrals on $\psi$ and $s_1,\ldots,s_m$ such that $U^{m,n,\vec b}_{\psi,\vec{s}}$ (see definition in Section 4) is bounded on the product of Morrey-Herz spaces. When reduce to the case of $U_\psi$, we will show that our sufficient condition is weaker than one obtained in \cite{TXZ}. 
\vskip12pt
 Our paper is organized as follows. In Section 2 we give necessary preliminaries on Morrey-Herz spaces and on a class of homogeneous weights. In Section 3, we prove the theorems on the bounds and boundedness of multilinear Hardy-Ces\`{a}ro operator in the product of Morrey-Herz spaces. In Section 4, we give a sufficient condition on weights such that the commutator of  $U^{m,n}_{\psi,\vec{s}}$, with symbols are Lipschitz, is bounded on the product of Morrey-Herz spaces.
 
\section{Preliminaries}
\vskip12pt
Throughout this paper $\omega(x)$ will be denote a nonnegative measurable function on $\mathbb R^d$, and $L^q_\omega(\mathbb R^d)$ be the space of all measurable functions $f$ such that
 \begin{equation*}
 \|f\|_{q,\omega}=\left( \int\limits_{\mathbb R^d}|f(x)|^q\omega(x)dx\right)^{1/q}<\infty.
\end{equation*}
\vskip1pt
In the following definitions $\chi_k=\chi_{C_k}; C_k=B_k\setminus B_{k-1}$ and $B_k=\{x\in \mathbb R^d: |x| \leq 2^k\}$, for $k\in \mathbb Z$, where $\chi_E$ is the characteristic function of a set $E$.
\begin{definition}\label{sec2def1} Let $\alpha \in \mathbb R, 0<p\leq \infty, 0< q\leq \infty.$ The weighted homogeneous Herz-type space $\dot{K}^{\alpha, p}_q(\omega)$ is defined by
\begin{equation}\label{sec2eq1}
\dot{K}^{\alpha, p}_q(\omega)=\{f\in L^q_{\text{loc}}(\mathbb R^d\setminus \{0\},\omega): \|f\|_{\dot{K}^{\alpha, p}_q(\omega)}<\infty\},
\end{equation}
where 
\begin{equation}\label{sec2eq2}
\|f\|_{\dot{K}^{\alpha, p}_q(\omega)}=\left\{\sum\limits_{k=-\infty}^{\infty}2^{k\alpha p}||f\chi_k||_{q,\omega}^p\right\}^{1/p}.
\end{equation}
(The usual modifications are made when $p=\infty$ and/or $q=\infty$.) 
\end{definition}

\begin{definition}\label{sec2def2} Let $\alpha \in \mathbb R$, $0<p \leq \infty$, $0<q\leq \infty$, $\lambda \geq 0$ and $\omega $ be non-negative weighted function. The homogeneous weighted Morrey-Herz-type space $M\dot{K}^{\alpha, \lambda}_{p, q}(\omega)$ is defined by
\begin{equation}\label{sec2eq3}
M\dot{K}^{\alpha, \lambda}_{p, q}(\omega)=\{f\in L^q_\text{loc}(\mathbb R^d\setminus\{0\},\omega): \|f\|_{M\dot{K}^{\alpha, \lambda}_{p, q}(\omega)}<\infty\},
\end{equation}
where
\begin{equation}\label{sec2eq4}
\|f\|_{M\dot{K}^{\alpha, \lambda}_{p, q}(\omega)}=\sup\limits_{k_0 \in \mathbb Z}2^{-k_0\lambda}\left\{\sum\limits_{k=-\infty}^{k_0}2^{k\alpha p}||f\chi_k||_{q,\omega}^p\right\}^{1/p},
\end{equation}
with the usual modifications made when $p=\infty$ or $q=\infty.$
\end{definition}

If $\omega\equiv1$, then we denote $\dot{K}^{\alpha, p}_q(\omega)$ and $M\dot{K}^{\alpha, \lambda}_{p, q}(\omega)$ respectively by $\dot{K}^{\alpha, p}_q(\mathbb R^d)$ and $M\dot{K}^{\alpha, \lambda}_{p, q}(\mathbb R^d)$, which are standard Herz spaces. Obviously, $\dot{K}^{0,p}_p(\mathbb R^d)=L^p (\mathbb R^d)$ for $0<p\leq\infty$; $\dot{K}^{\frac{\alpha}{p},p}_p(\mathbb R^d)= L^p(|x|^{\alpha}dx) (\mathbb R^d)$ for all  $0<p\leq \infty$ and $\alpha\in \mathbb R.$ Meanwhile, $M\dot{K}^{\alpha,0}_{p,q}(\mathbb R^d)=\dot{K}^{\alpha,p}_q(\mathbb R^d)$, so the special cases of Morrey-Herz spaces are Herz spaces. 
\vskip12pt
We would like to recall the definition of the class of homogeneous weights introduced by \cite{CH}.
\begin{definition}\label{def3} Let $\gamma$ be a real number. Let $\mathcal W_{\gamma}$ be the set of all functions $\omega$ on $\mathbb R^d$, which are measurable, $\omega(x)>0$ for almost everywhere $x\in \mathbb R^d$, $0<\int\limits_{S_d}\omega(y)\sigma(y)<\infty$, and are absolutely homogeneous of degree $\gamma$, that is $\omega(tx)=|t|^{\gamma}\omega(x)$, for all $t\in \mathbb R\setminus \{0\}, x\in \mathbb R^d.$
\end{definition}
We remark that $\mathcal W=\bigcup\limits_{\gamma}\mathcal W_\gamma$ contains strictly the set of power weights $\omega(x)=|x|^\gamma$. For further discussions, we refer to \cite{CH}.
\vskip12pt
Throughout the whole paper, $S_d=\{x\in \mathbb R^d : |x|=1\}$ and we also denote $S_d=\frac{2\pi^{\frac d 2}}{\Gamma(\frac d 2)}$ . By $\omega$ we will denote a weight from $\mathcal W_\gamma$, where $\gamma>-d$. We also denote by $\psi$ a nonnegative and measurable function on $[0,1]^n$. 

\begin{definition}\label{sec4def1} Suppose that $0<\beta<1$, the Lipschitz space $Lip^{\beta}(\mathbb R^n)$ is defined as the set of all functions $f: \mathbb R^n \to \mathbb C$ such that 
\begin{equation}\label{sec2eq5}
||f||_{Lip^{\beta}(\mathbb R^n)}:=\sup\limits_{x,y\in \mathbb R^n, x \ne y}\frac{|f(x)-f(y)|}{|x-y|^{\beta}}<\infty.
\end{equation}
\end{definition}
\vskip12pt
If $(X,\|\cdot\|),(X_1,\|\cdot\|_1),\ldots,(X_m,\|\cdot\|_m)$ are normed spaces and $T$ is a sublinear operator $T:X_1\times\cdots\times X_m\to X$ then we set
\begin{equation}\label{sec2eq6}
\|T\|_{X_1\times\cdots\times X_m\to X}=\sup\limits_{\|x_1\|_1\leq1,\ldots,\|x_m\|_m\leq1}\|T(x_1,\ldots,x_m)\|.
\end{equation}
\section{Boundedness of $U^{m,n}_{\psi,\vec{s}}$ on the product of Herz and Morrey-Herz spaces}
In this section we will state and prove the results on the boundedness and bounds of the multilinear Hardy-Ces\`{a}ro operators on the product of Herz spaces and Morrey-Herz spaces. Before state our main results, let us introduction some notations. 
\vskip12pt
Throughout this section, $\beta>0,\gamma,\alpha, \alpha_1,\ldots, \alpha_m$ are real numbers, $\gamma_1,\ldots, \gamma_m>-d$, $0<p<\infty$, $1\leq q<\infty$, $1\leq p_i, q_i <\infty$ for $i=1,\ldots,m$ and $\lambda,\lambda_1,\ldots,\lambda_m\ge0$. All such constants satisfy the following relations
$$ \alpha_1+\alpha_2+\cdots+\alpha_m=\alpha, $$
$$ \frac 1{p_1}+\frac1 {p_2}+\cdots+\frac1{p_m}=\frac1 p, $$
$$ \frac 1 {q_1}+\frac1{q_2}+\cdots+\frac1{q_m}=\frac 1 q,$$
$$\frac{\gamma_1}{q_1}+\frac{\gamma_2}{q_2}+\cdots+\frac{\gamma_m}{q_m}=\frac\gamma q,$$
$$\lambda_1+\lambda_2+\cdots+\lambda_m=\lambda.$$
Functions $\omega_i$ belong to $\mathcal W_{\gamma_i}$ for all $i=1,\ldots,m$, and we set
\begin{equation}\label{sec3eq1}
\omega(x)= \prod_{i=1}^m\omega^{\frac q{q_i}}_i(x).
\end{equation}
Obviously that $\omega\in \mathcal W_{\gamma}$. Such weights $\omega$ as defined in (\ref{sec3eq1}) arise naturally in the theory  of multilinear operators.  
\vskip 12pt
The main results in this section are as follows


\begin{theorem}\label{sec3theo2}

{\rm(i)} Let $s_1(t),\ldots,s_m(t)\neq0$ almost everywhere in $[0,1]^n$ and 
\begin{equation}\label{sec3eq2}
{\mathcal A_1}=\int\limits_{[0,1]^n}\left(\prod_{i=1}^m|s_i(t)|^{-\alpha_i-\frac {d+\gamma_i}{q_i}+\lambda_i}\right)\psi(t)dt<\infty.
\end{equation}
Suppose that $1\leq p<\infty$ or $0<p<1$ and at least one of $\lambda_1,\ldots,\lambda_m$ is positive. Then 
\begin{equation}\label{sec3eq3}
\|U^{m,n}_{\psi,\vec{s}}(f_1,\ldots,f_m)\|_{M\dot{K}^{\alpha, \lambda}_{p,q}(\omega)}\leq C_{\vec{\alpha},\vec{\lambda}}\cdot {\mathcal A}_1\cdot\prod\limits_{i=1}^m\|f_i\|_{M\dot{K}^{\alpha_i, \lambda_i}_{p_i,q_i}(\omega_i)}.
\end{equation}
Here
\[
C_{\vec{\alpha},\vec{\lambda}}=\begin{cases}\prod\limits_{k=1}^m\left(2^{|\alpha_k-\lambda_k|}+1\right)\qquad\qquad\qquad\;\text{if}\;\;1\leq p<\infty&\\ \dfrac{2^\lambda}{\left(2^{\lambda p}-1\right)^{1/p}}\prod\limits_{k=1}^m\left(2^{|\alpha_k-\lambda_k|}+1\right)\quad\text{if}\;\;0<p<1\;\text{and}\;\lambda>0.\end{cases}
\]

{\rm(ii)} Conversely, let $0<p<\infty$, $0<\lambda_i<\infty$ for $i=1,\ldots,m$. Suppose that  $U^{m,n}_{\psi,\vec{s}}$ is defined as a bounded operator from $M\dot{K}^{\alpha_1, \lambda_1}_{p_1,q_1}(\omega_1)\times \cdots  \times M\dot{K}^{\alpha_m, \lambda_m}_{p_m,q_m}(\omega_m)$ to $M\dot{K}^{\alpha, \lambda}_{p,q}(\omega)$. Then we have that (\ref{sec3eq1}) holds and
\begin{equation}\label{sec3eq4}
\|U^{m,n}_{\psi,\vec{s}}\|_{M\dot{K}^{\alpha_1, \lambda_1}_{p_1,q_1}(\omega_1)\times\cdots\times M\dot{K}^{\alpha_m, \lambda_m}_{p_m,q_m}(\omega_m) \to M\dot{K}^{\alpha, \lambda}_{p,q}(\omega)}\geq {\mathcal A_1}\cdot D_{\vec{\alpha},\vec{\lambda}},
\end{equation}
where 
\[
D_{\vec{\alpha},\vec{\lambda}}=\frac{\prod\limits_{i=1}^m(2^{\lambda_i p_i}-1)^{1/p_i}}{(2^{\lambda p}-1)^{1/p}}\cdot\dfrac{\left(1-2^{-q(\lambda-\alpha)}\right)^{1/q}}{\prod\limits_{i=1}^m(1-2^{-q_i(\lambda_i-\alpha_i)})^{1/q_i}}\cdot\dfrac{\prod\limits_{i=1}^m(q_i(\lambda_i-\alpha_i))^{1/q_i}}{(q(\lambda-\alpha))^{1/q}}\cdot\dfrac{\left(\omega(S_d)\right)^{1/q}}{\prod\limits_{i=1}^m\left(\omega_i(S_d)\right)^{1/q_i}}.
\]

\end{theorem}

\begin{theorem}\label{sec3theo1}
{\rm(i)} 
If $1\leq p<\infty$, $s_1(t),\ldots,s_m(t)\neq0$ almost everywhere in $[0,1]^n$ and
\begin{equation}\label{sec3eq5}
{\mathcal A}_2=\int\limits_{[0,1]^n}\left(\prod_{i=1}^m|s_i(t)|^{-\frac {d+\gamma_i}{q_i}-\alpha_i}\right)\psi(t)dt<\infty,
\end{equation}
then 
\begin{equation}\label{sec3eq6}
\|U^{m,n}_{\psi,\vec{s}}(f_1,\ldots,f_m)\|_{\dot{K}^{\alpha, p}_{q}(\omega)}\leq{\mathcal A}_2\cdot \prod\limits_{k=1}^m\left(2^{|\alpha_k|}+1\right)\cdot\prod\limits_{i=1}^m\|f_i\|_{\dot{K}^{\alpha_i, p_i}_{q_i}(\omega_i)}
\end{equation}

{\rm(ii)}  Let assume that $|s_i(t_1,\ldots,t_m)|\ge \min\{t_1^\beta,\ldots,t_m^\beta\}$ for $i=1,\ldots,m$ and $U^{m,n}_{\psi,\vec{s}}$ is bounded from $\dot{K}^{\alpha_1, p_1}_{q_1}(\omega_1)\times\cdots\times\dot{K}^{\alpha_m, p_m}_{q_m}(\omega_m)$ to $\dot{K}^{\alpha, p}_{q}(\omega)$. Then (\ref{sec3eq5}) holds and 
\begin{equation}\label{sec3eq7}
\|U^{m,n}_{\psi,\vec{s}}\|_{\dot{K}^{\alpha_1, p_1}_{q_1}(\omega_1)\times\cdots\times\dot{K}^{\alpha_m, p_m}_{q_m}(\omega_m)\to\dot{K}^{\alpha, p}_{q}(\omega)} \ge {\mathcal A}_2 \cdot E_{\vec\alpha}.
\end{equation}
Where 
\[
E_{\vec\alpha}=\dfrac{(mp)^{1/p}}{\prod\limits_{i=1}^mp_i^{1/p_i}}\cdot \left(\dfrac{2^{q\alpha}-1}{q\alpha}\right)^{1/q}\cdot\prod\limits_{i=1}^m\left(\dfrac{q_i\alpha_i}{2^{q_i\alpha_i}-1}\right)^{1/q_i}\cdot\dfrac{\left(\omega(S_d)\right)^{1/q}}{\prod\limits_{i=1}^m\left(\omega_i(S_d)\right)^{1/q_i}}.
\]

\end{theorem}
When $\alpha_1=\cdots=\alpha_m=0$ we obtain the boundedness and bounds for multilinear Hardy-Ces\`{a}ro operator on the product Lebesgue spaces. However, the results are worse than those obtained in \cite{HK}. In fact, Hung and Ky \cite[Theorem 3.1]{HK} proved that the norm of $U^{m,n}_{\psi,\vec{s}}$ from $L^{p_1}_{\omega_1}\times\cdots\times L^{p_m}_{\omega_m}$ to $L^p_{\omega}$ is exactly $\int\limits_{[0,1]^n}\left(\prod\limits_{i=1}^m|s_i(t)|^{-\frac {d+\gamma_i}{q_i}-\alpha_i}\right)\psi(t)dt.$ 
\begin{proof}[Proof of Theorem \ref{sec3theo2}]
Suppose that $s_1(t),\ldots,s_m(t)\neq0$ a.e $t\in[0,1]^n$ such that (\ref{sec3eq1}) holds. Fix an $k\in\mathbb Z$ and consider functions $f_i\in M\dot{K}^{\alpha_i, \lambda_i}_{p_i,q_i}(\omega_i)$ for $i=1,\ldots,m$. By H\"{o}lder, Minkowski inequalities
\begin{align*}
&\|U^{m,n}_{\psi,\vec{s}}(f_1, \ldots,f_m)\chi_k\|_{q,\omega}\\
\leq\;& \int_{[0,1]^n} \prod_{i=1}^m \left(\int_{C_k}\left|f_i(s_i(t)x)\right|^{q_i}\omega_i(x)dx\right)^{1/{q_i}}\psi(t)dt \\
\leq\;&\int_{[0,1]^n} \prod_{i=1}^m\left(\int_{s_i(t)C_k}\left|f_i(x)\right|^{q_i}|s_i(t)|^{-d-\gamma_i}\omega_i(x)dx\right)^{1/{q_i}}\psi(t)dt\\
\leq\;&\int\limits_{[0,1]^n}\left(\prod_{i=1}^m|s_i(t)|^{-\frac {d+\gamma_i}{q_i}}\right)\cdot\prod_{i=1}^m\left(\int\limits_{s_i(t)C_k}|f_i(x)|^{q_i}\omega_i(x)dx\right)^{1/q_i}\cdot\psi(t)dt.
\end{align*}
For each $t$ such that $\left|s_1(t)\cdots  s_m(t)\right|>0$, there exists integer numbers $\ell_1,\ldots,\ell_m$ such that $2^{\ell_i-1}<|s_i(t)|\leq2^{\ell_i}$ for $i=1,\ldots,m$. After setting $\psi_{\vec s}(t)=\left(\prod\limits_{i=1}^m|s_i(t)|^{-\frac {d+\gamma_i}{q_i}}\right)\psi(t)$, we have that
\begin{align*}
&\|U^{m,n}_{\psi,\vec{s}}(f_1,\ldots, f_m)\chi_k\|_{q,\omega}\\
\leq\; &\int\limits_{[0,1]^n}\prod_{i=1}^m\left(\int\limits_{C_{k+\ell_i-1}\cup C_{k+\ell_i}}|f_i(x)|^{q_i}\omega_i(x)dx\right)^{1/q_i}\psi_{\vec s}(t)dt\\
\leq\;&\int\limits_{[0,1]^n}\prod_{i=1}^m\left(\|f_i\chi_{k+\ell_i-1}\|_{q_i,\omega_i}+\|f_i\chi_{k+\ell_i}\|_{q_i,\omega_i}\right)\psi_{\vec s}(t)dt.
\end{align*}
Thus,
\begin{equation}\label{sec3eq8}
\|U^{m,n}_{\psi,\vec{s}}(f_1,\ldots, f_m)\chi_k\|_{q,\omega}\leq\int\limits_{[0,1]^n}\prod_{i=1}^m\left(\|f_i\chi_{k+\ell_i-1}\|_{q_i,\omega_i}+\|f_i\chi_{k+\ell_i}\|_{q_i,\omega_i}\right)\psi_{\vec s}(t)dt.
\end{equation}
Now we consider the following cases:
\vskip12pt
{\bf Case 1:} Suppose that $1\leq p<\infty$. To estimate the norm of $U^{m,n}_{\psi,\vec{s}}(f_1, \ldots,f_m)$ in $M\dot{K}^{\alpha, \lambda}_{p, q}(\omega)$ space, we shall require the Minkowski integral inequality in the following form.
\begin{lemma}\label{sec3lem1} 
Let $p\ge1$ and $(f_k)_{k\ge1}$ be nonnegative and measurable functions on $[0,1]^n$. Then
\[
\sum\limits_{k=1}^\infty\left(\int_{[0,1]^n}f_k(t)dt\right)^p\leq \left(\int\limits_{[0,1]^n}\left(\sum\limits_{k=1}^\infty f_k^p(t)\right)^{1/p}dt\right)^p
\] 
\end{lemma}
Since the proof of Lemma \ref{sec3lem1} is straightforward by using the characterization of $\ell^p$-norm, it is left to the reader. Consequently, by the definition of $U^{m,n}_{\psi,\vec s}$, H\"{o}lder's inequality, Lemma \ref{sec3lem1} and (\ref{sec3eq3}) one has
\begin{align*}
&\|U^{m,n}_{\psi,\vec{s}}(f_1, \ldots,f_m)\|_{M\dot{K}^{\alpha, \lambda}_{p, q}(\omega)}\\
=\;&\sup\limits_{k_0\in \mathbb Z}2^{-k_0\lambda}\left(\sum_{k=-\infty}^{k_0}2^{k\alpha p}\|U^{m,n}_{\psi,\vec{s}}(f_1, \ldots,f_m)\chi_k\|_{q,\omega}^p\right)^{1/p}\\
\leq\;& \sup\limits_{k_0\in \mathbb Z}2^{-k_0\lambda}\left(\sum_{k=-\infty}^{k_0}2^{k\alpha p}\left(\int\limits_{[0,1]^n}\prod_{i=1}^m\left(\|f_i\chi_{k+\ell_i-1}\|_{q_i,\omega_i}+\|f_i\chi_{k+\ell_i}\|_{q_i,\omega_i}\right)\psi_{\vec s}(t)dt\right)^p\right)^{1/p}\\
\leq\;&\sup\limits_{k_0\in \mathbb Z}2^{-k_0\lambda}\int\limits_{[0,1]^n}\left(\sum\limits_{k=-\infty}^{k_0}2^{k\alpha p}\prod_{i=1}^m\left(\|f_i\chi_{k+\ell_i-1}\|_{q_i,\omega_i}+\|f_i\chi_{k+\ell_i}\|_{q_i,\omega_i}\right)^p \right)^{1/p}\psi_{\vec s}(t)dt\\
\leq\;&\sup\limits_{k_0\in \mathbb Z}2^{-k_0\lambda}\int_{[0,1]^n}\prod_{i=1}^m\left(\sum\limits_{k=-\infty}^{k_0}2^{k\alpha_i p_i}\left(\|f_i\chi_{k+\ell_i-1}\|_{q_i,\omega_i}+\|f_i\chi_{k+\ell_i}\|_{q_i,\omega_i}\right)^{p_i} \right)^{1/p_i}\psi_{\vec s}(t)dt\\
\leq\;&\sup\limits_{k_0\in \mathbb Z}\int\limits_{[0,1]^n}\prod_{i=1}^m2^{-k_0\lambda_i}\left\{\left(\sum_{k=-\infty}^{k_0}2^{k\alpha_i p_i}||f_i\chi_{k+\ell_i-1}||_{q_i,\omega_i}^{p_i}\right)^{1/p_i}+\left(\sum_{k=-\infty}^{k_0}2^{k\alpha_i p_i}||f_i\chi_{k+\ell_i}||_{q_i,\omega_i}^{p_i}\right)^{1/p_i} \right\}\psi_{\vec s}(t)dt\\
\leq\;&\prod_{i=1}^m\|f_i\|_{M\dot{K}^{\alpha_i, \lambda_i}_{p_i, q_i}(\omega_i)}\int\limits_{[0,1]^n}\prod_{i=1}^m\left(2^{-(\ell_i-1)(\alpha_i-\lambda_i)}+2^{-\ell_i(\alpha_i-\lambda_i)}\right)\psi_{\vec s}(t)dt\\
\leq\;&\prod\limits_{i=1}^m\left(2^{|\alpha_i-\lambda_i|}+1\right)\prod\limits_{i=1}^m\|f_i\|_{M\dot{K}^{\alpha_i, \lambda_i}_{p_i, q_i}(\omega_i)}\int\limits_{[0,1]^n}\left(\prod_{i=1}^m|s_i(t)|^{-\alpha_i-\frac {d+\gamma_i}{q_i}+\lambda_i}\right)\psi(t)dt.
\end{align*}
Therefore,
\[
\|U^{m,n}_{\psi,\vec{s}}(f_1, \ldots,f_m)\|_{M\dot{K}^{\alpha, \lambda}_{p, q}(\omega)}\leq C_{\vec{\alpha},\vec{\lambda}}\cdot{\mathcal A}_1\cdot\prod\limits_{i=1}^m\|f_i\|_{M\dot{K}^{\alpha_i, \lambda_i}_{p_i, q_i}(\omega_i)}.
\]
It means that $U^{m,n}_{\psi,\vec{s}}$ is bounded from $M\dot{K}^{\alpha_1, \lambda_1}_{p_1, q_1}(\omega_1)\times\cdots\times M\dot{K}^{\alpha_m, \lambda_m}_{p_m, q_m}(\omega_m)$ to $ M\dot{K}^{\alpha, \lambda}_{p, q}(\omega)$ and its norm is not greater than $ C_{\vec{\alpha},\vec{\lambda}}\cdot{\mathcal A}_1$.
\vskip12pt
{\bf Case 2:} Suppose that $0<p<1$ and $\lambda_1,\ldots,\lambda_m$ are not equal to zero. We shall need the following lemma.
\begin{lemma}\label{sec3lem2}
If $f\in M\dot{K}^{\alpha,\lambda}_{p,q}(\omega)$ then $\|f\chi_k\|_{q,\omega}\leq 2^{k(\lambda-\alpha)}\|f\|_{M\dot{K}^{\alpha,\lambda}_{p,q}(\omega)}$. 
\end{lemma}
The proof of Lemma \ref{sec3lem2} follows directly from Definition \ref{sec2def2}, thus we omit it here. By (\ref{sec3eq9}) and Lemma \ref{sec3lem2}, we have 
\begin{align*}
&\|U^{m,n}_{\psi,\vec{s}}(f_1, \ldots,f_m)\|_{M\dot{K}^{\alpha, \lambda}_{p, q}(\omega)}\\
\leq\;&\sup\limits_{k_0\in \mathbb Z}2^{-k_0\lambda} \left(\sum_{k=-\infty}^{k_0}2^{kp\alpha}\left(\int_{[0,1]^n}\prod\limits_{i=1}^m\left(\|f_i\chi_{k+\ell_i-1}\|_{q_i,\omega_i}+\|f_i\chi_{k+\ell_i}\|_{q_i,\omega_i}\right)\psi_{\vec s}(t)dt\right)^p\right)^{1/p}\\
\leq\;&\sup\limits_{k_0\in \mathbb Z}2^{-k_0\lambda}\left(\sum_{k=-\infty}^{k_0}2^{kp\alpha}\left(\int_{[0,1]^n}\prod\limits_{i=1}^m\left(2^{(k+\ell_i-1)(\lambda_i-\alpha_i)}+2^{(k+\ell_i)(\lambda_i-\alpha_i)}\right)\|f_i\|_{M\dot{K}^{\alpha_i,\lambda_i}_{p_i,q_i}(\omega_i)}\psi_{\vec s}(t)dt\right)^p\right)^{1/p}\\
\leq\;&\prod\limits_{i=1}^m\left(2^{|\alpha_i-\lambda_i|}+1\right)\left(\prod\limits_{i=1}^m\|f_i\|_{M\dot{K}^{\alpha_i, \lambda_i}_{p_i, q_i}(\omega_i)}\right)\left(\int\limits_{[0,1]^n}\prod\limits_{i=1}^m|s_i(t)|^{\lambda_i-\alpha_i}\psi_{\vec s}(t)dt\right)\sup_{k_0\in\mathbb Z}\left(\sum\limits_{k=-\infty}^{k_0}2^{(k-k_0)\lambda p}\right)^{1/p}.
\end{align*}
We here remind that $\lambda=\lambda_1+\cdots+\lambda_m$ and $\alpha=\alpha_1+\cdots+\alpha_m$.  Since $\lambda>0$, the series $\sum\limits_{k=-\infty}^{k_0}2^{(k-k_0)\lambda p}$ is convergent and its sum is $\dfrac{2^{\lambda p}}{2^{\lambda p}-1}$. Therefore
\[
\|U^{m,n}_{\psi,\vec{s}}(f_1, \ldots,f_m)\|_{M\dot{K}^{\alpha, \lambda}_{p, q}(\omega)}\leq C_{\vec{\alpha},\vec{\lambda}}\cdot{\mathcal A}_1\cdot\left(\prod\limits_{i=1}^m\|f_i\|_{M\dot{K}^{\alpha_i, \lambda_i}_{p_i, q_i}(\omega_i)}\right),
\]
as desired.
\vskip12pt
Now, we assume that $U^{m,n}_{\psi,\vec{s}}$ is defined as a bounded operator from $M\dot{K}^{\alpha_1, \lambda_1}_{p_1,q_1}(\omega_1)\times \cdots  \times M\dot{K}^{\alpha_m, \lambda_m}_{p_m,q_m}(\omega_m)$ to $M\dot{K}^{\alpha, \lambda}_{p,q}(\omega)$. Hence,
\begin{equation}\label{sec3eq9}
\|U^{m,n}_{\psi,\vec{s}}(f_1,\ldots,f_m)\|_{ M\dot{K}^{\alpha, \lambda}_{p,q}(\omega)}\leq \|U^{m,n}_{\psi,\vec{s}}\|_{M\dot{K}^{\alpha_1, \lambda_1}_{p_1,q_1}(\omega_1)\times\cdots\times M\dot{K}^{\alpha_m, \lambda_m}_{p_m,q_m}(\omega_m) \to M\dot{K}^{\alpha, \lambda}_{p,q}(\omega)} \prod\limits_{i=1}^m\|f_i\|_{M\dot{K}^{\alpha_i, \lambda_i}_{p_i, q_i}(\omega_i)}.
\end{equation}
 
For each $i=1,\ldots,m$ we set
\begin{equation}\label{sec3eq10} 
f_i(x)=|x|^{-\alpha_i-\frac{d+\gamma_i}{q_i}+\lambda_i}.
\end{equation} 
Then, it is not hard to see that
\begin{equation}\label{sec3eq12}
\|f_i\|_{M\dot{K}^{\alpha_i, \lambda_i}_{p_i, q_i}(\omega_i)}=\dfrac{2^{\lambda_i}}{(2^{\lambda_i p_i}-1)^{1/p_i}}\cdot\left(\dfrac{1-2^{-q_i(\lambda_i-\alpha_i)}}{q_i(\lambda_i-\alpha_i)}\right)^{1/q_i}\cdot\left(\omega_i(S_d)\right)^{1/q_i},
\end{equation}
and
\[
U^{m,n}_{\psi,\vec s}(f_1,\ldots,f_m)(x)=|x|^{-\alpha+\lambda-\frac{d+\gamma}q}\int_{[0,1]^n}\prod\limits_{i=1}^m|s_i(t)|^{-\alpha_i+\lambda_i-\frac{d+\gamma_i}{q_i}}\psi(t)dt.
\]
This implies  that
\begin{equation}\label{sec3eq13}
\|U^{m,n}_{\psi,\vec s}(f_1,\ldots,f_m)\|_{M\dot{K}^{\alpha,\lambda}_{p,q}(\omega)}={\mathcal A}\cdot \frac{2^\lambda}{(2^{\lambda p}-1)^{1/p}}\cdot\left(\dfrac{1-2^{-q(\lambda-\alpha)}}{q(\lambda-\alpha)}\right)^{1/q}\cdot\left(\omega(S_d)\right)^{1/q}.
\end{equation}
Therefore, by (\ref{sec3eq9}), (\ref{sec3eq12}) and (\ref{sec3eq13}), we get (\ref{sec3eq7}). This completes the proof of Theorem \ref{sec3theo2}.
\end{proof}

\begin{proof}[Proof of Theorem \ref{sec3theo1}] Since Morrey-Herz spaces $M\dot{K}^{\alpha,\lambda}_{p,q}(\omega)$ reduces to Herz spaces $\dot{K}^{\alpha,p}_q(\omega)$, thus part (i) of Theorem \ref{sec3theo1} is a special case of Theorem \ref{sec3theo2}, part (i). To prove the second part, we suppose that $U^{m,n}_{\psi,\vec{s}}$ is bounded from $\dot{K}^{\alpha_1, p_1}_{q_1}(\omega_1)\times\cdots\times\dot{K}^{\alpha_m, p_m}_{q_m}(\omega_m)$ to $\dot{K}^{\alpha, p}_{q}(\omega)$. Fix $0<\epsilon<1$, for $i=1,\ldots,m$ we set
\[
f_i(x)=\begin{cases}\;0\qquad\qquad\qquad\,\;\;\text{if}\quad |x|\leq1&\\ |x|^{-\alpha_i-\frac{d+\gamma_i}{q_i}-\epsilon}\qquad\text{if}\quad|x|>1.\end{cases}
\]
Then, a simple computation gives $f_i\in\dot{K}^{\alpha_i,p_i}_{q_i}(\omega_i)$ and 
\[
\|f_i\|_{\dot{K}^{\alpha_i,p_i}_{q_i}(\omega_i)}=\dfrac1{(2^{\epsilon p_i}-1)^{1/p_i}}\cdot\left(\dfrac{2^{q_i(\alpha_i+\epsilon)}-1}{q_i(\alpha_i+\epsilon)}\right)^{1/q_i}\cdot \left(\omega_i(S_d)\right)^{1/q_i}.
\]
We see that 
\[
U^{m,n}_{\psi,\vec s}(f_1,\ldots,f_m)(x)=|x|^{-\alpha-\frac{d+\gamma}q-m\epsilon}\int\limits_{E_x}\prod\limits_{i=1}^m|s_i(t)|^{-\alpha_i-\frac{d+\gamma_i}{q_i}-\epsilon}\psi(t)dt,
\]
here $E_x$ is the set of $t\in[0,1]^n$ such that $|s_i(t)x|>1$. By the assumption of the hypothesis, $|s_i(t_1,\ldots,t_n)|\geq\min\{t_1,\ldots,t_n\}^{\beta}$, it gives $[1/|x|^{1/\beta},1]\subset E_x$ for any $|x|>1$. Let $k_0$ be the smallest integer number such that $2^{(1-k)/\beta}<\epsilon$ for any $k\ge k_0$. Then
\begin{align*}
&\|U^{m,n}_{\psi,\vec s}(f_1,\ldots,f_m)\|^p_{\dot{K}^{\alpha, p}_{q}(\omega)}\\
\ge\;&\sum\limits_{k=k_0}^\infty 2^{k\alpha p}\left(\int_{2^{k-1}<|x|\leq 2^k}|x|^{-(\alpha+m\epsilon)q-(d+\gamma)}\omega(x)\left(\int_{[2^{(1-k)/\beta},1]^n}\prod\limits_{i=1}^m|s_i(t)|^{-\alpha_i-\frac{d+\gamma_i}{q_i}-\epsilon}\psi(t)dt\right)^qdx\right)^{p/q} \\
\ge\;&\left(\sum\limits_{k=k_0}^{\infty}2^{-kpm\epsilon}\right)\cdot\left(\dfrac{2^{q(\alpha+m\epsilon)}-1}{q(\alpha+m\epsilon)}\right)^{p/q}\cdot(\omega(S_d))^{p/q}\cdot\left(\int_{[\epsilon,1]^n}\prod\limits_{i=1}^m|s_i(t)|^{-\alpha_i-\frac{d+\gamma_i}{q_i}-\epsilon}\psi(t)dt\right)^p.
\end{align*}
On the other hand
\[
\prod\limits_{i=1}^m\|f_i\|_{\dot{K}^{\alpha_i,p_i}_{q_i}(\omega_i)}=\prod\limits_{i=1}^m\left(\dfrac1{(2^{\epsilon p_i}-1)^{1/p_i}}\cdot\left(\dfrac{2^{q_i(\alpha_i+\epsilon)}-1}{q_i(\alpha_i+\epsilon)}\right)^{1/q_i}\cdot \left(\omega_i(S_d)\right)^{1/q_i}\right).
\]
Thus, letting $\epsilon\to0$, by the Lebesgue's dominated convergence Theorem and the obove estimate give (\ref{sec3eq5}). This completes the proof of Theorem \ref{sec3theo1}.
\end{proof}

\section{Commutator of weighted multilinear Hardy-Ces\`{a}ro operator}
In 1976, a famous result of Coifman, Rochberg and Weiss \cite{CRW} stated that if $b\in BMO(\mathbb R^d)$ then the commutator $[b,T]$ is bounded on $L^p(\mathbb R^d)$ for every $p\in(1,\infty)$ and $T$ is a classical Calder\'{o}n-Zygmund operator. Moreover $BMO(\mathbb R^d)$ is the largest space having this property.  Similarly to the classical result of Coifman-Rochberg-Weiss, Z.W. Fu, Z.G. Liu and S.Z. Lu \cite{FLL} proved that if $b\in BMO(\mathbb R^d)$ then $[b,U_\psi]$ is bounded on $L^p(\mathbb R^d)$ if and only if 
\begin{equation}\label{sec4eq1a}
\int\limits_0^1t^{-d/p}\psi(t)\log\frac2tdt<\infty.
\end{equation}
From conditions (\ref{sec1eq4}) and (\ref{sec4eq1a}), we observe that the commutator $[b,U_\psi]$ is more singular than $U_\psi$. Such results were extended by Z. Fu, S. Lu \cite{FL1} to Morrey spaces, by N.M. Chuong, H.D. Hung \cite{CH} in case $T$ is Hardy-Ces\`{a}ro operators. In \cite{HK}, H.D. Hung and L.D. Ky obtained sufficient and necessary conditions on weight functions so that the commutator of these weighted multilinear Hardy-Ces\`{a}ro operator (with symbols in central BMO space) are bounded on the product of central Morrey spaces. 
\vskip12pt
Notice that there are several results on commutator for other type of Hardy operators or even for some class of Hausdorff operators (for example see \cite{CH,FL2,FL1,FL,GY,HK} and references therein). But there is only partial such results for the commutator operator $[b,U_\psi]$, with symbols are Lipschitz, in Morrey-Herz space by Tang-Xue-Zhou \cite{TXZ}. In this section, we will improve this result and extend it to the operator $U^{m,n}_{\psi,\overrightarrow{s}}$. At first, let us remind the definition of the multilinear commutator generated by $U^{m,n}_{\psi,\overrightarrow{s}}$.
\begin{definition}\label{sec4def3}Let $m,n\in\mathbb N$, $\psi:[0,1]^n\to[0,\infty)$, $s_1,\ldots,s_m:[0,1]^n\to\mathbb R$, $b_1,\ldots,b_m$ be locally integrable functions on $\mathbb R^d$ and $f_1,\ldots,f_m:\mathbb R^d\to\mathbb C$ be measurable functions. The commutator of weighted multilinear Hardy-Ces\`{a}ro operator $U^{m,n}_{\psi,\overrightarrow{s}}$ is defined as
\begin{equation}\label{sec4eq1}
U^{m,n,\overrightarrow{b}}_{\psi,s}\left(f_1,\ldots,f_m\right)(x):=\int_{[0,1]^n}\left(\prod\limits_{k=1}^mf_k\left(s_k(t)x\right)\right)\left(\prod_{k=1}^m\left(b_k(x)-b_k\left(s_k(t)x\right)\right)\right)\psi(t)dt.
\end{equation}
\end{definition}

Our main result in this section is as follows.

\begin{theorem}\label{sec4theo1} Let $\alpha_i>-d$, $1 \leq  q\leq q_i<\infty$, $0\leq r_i$, $1\leq p_i<\infty$, $0<\beta_i<1$, $0<\beta<1$, $0\leq \lambda_i,\lambda\leq1$ for $i=1,\ldots,m$ with $\alpha=\alpha_1+\cdots+\alpha_m>-d$, $\beta=\beta_1+\cdots+\beta_m$, $\lambda=\lambda_1+\cdots+\lambda_m$,  $\frac 1 p=\frac 1{p_1}+\cdots+\frac 1{p_m}$, and $\frac1q=\frac 1{q_1}+\cdots+\frac 1{q_m}+\frac 1{r_1}+\cdots+\frac 1{r_m}$. Suppose that $b_i\in Lip^{\beta_i} $ and $\omega_i$ as in (\ref{sec3eq1}) for $i=1,\ldots,m$. Functions $s_1(t),\ldots,s_m(t)\neq0$ almost everywhere $t\in[0,1]^n$ such that 
\begin{equation}\label{sec4eq3}
\int_{[0,1]^n}\left(\prod_{i=1}^m|s_i(t)|^{-\frac{d+\gamma_i}{q_i}+\lambda_i-\alpha_i}|1-s_i(t)|^{\beta_i}\right)\psi(t)dt <\infty
\end{equation}
Then the commutator  $U^{m,n,\vec{b}}_{\psi,\vec{s}}$ is determined as a bounded operator from $M\dot{K}^{\alpha_1, \lambda_1}_{p_1, q_1}(\omega_1)\times \cdots \times M\dot{K}^{\alpha_m, \lambda_m}_{p_m, q_m}(\omega_m)$ to $M\dot{K}^{\alpha^\prime, \lambda}_{p, q}(\omega)$ when $0<p<1$ and  $\lambda > 0$ or when $1\leq p<\infty$ and $\lambda\ge0$. Here  
\begin{equation}\label{sec4eq4}
\alpha^\prime=\alpha-\sum\limits_{i=1}^m\beta_i-\sum\limits_{i=1}^m\frac{d+\gamma_i}{r_i}. 
\end{equation}
\end{theorem}
When $m=n=1$,  $\omega_1=\frac1{|B_0|}$, $s_1(t)\equiv t$, let $U^b_\psi=U^{1,1,\vec{b}}_{\psi,\vec{s}}$, we obtain the following result.
\begin{corollary}
Let $\psi: [0; 1]\to[0;\infty)$ be a measurable function, $0<\beta<1$, $b\in \text{Lip}^\beta(\mathbb R^d)$, $1\leq q_2\leq q_1<\infty$. If 
\begin{equation}\label{sec4eq5}
{\mathcal A}=\int_0^1t^{-\left(\gamma_1-\lambda-\frac d{q_1}\right)}(1-t)\psi(t)dt<\infty,
\end{equation}
then $U^b_\psi$ is bounded from $M\dot{K}^{\alpha_1,\lambda}_{p,q_1}$ to $M\dot{K}^{\alpha_2,\lambda}_{p,q_2}$, where $\alpha_1=\alpha_2+\beta+d\left(\dfrac1{q_2}-\dfrac1{q_1}\right)$.
\end{corollary}
In \cite{TXZ}, to obtain the boundedness of $U^b_\psi$ from $M\dot{K}^{\gamma_1,\lambda}_{p,q_1}$ to $M\dot{K}^{\gamma_2,\lambda}_{p,q_2}$, the authors required a sufficient condition condition on $\psi$ that
\[
\mathcal C=\int_0^1t^{-\left(\gamma_1-\lambda-\frac d{q_1}\right)}\psi(t)dt<\infty.
\] 
Since $0\leq t\leq 1$, then $\mathcal A\leq \mathcal C$. In fact by choosing $\psi(t)=\frac t{1-t}$, $\gamma_1-\lambda-\frac d{q_1}=1$ then it is easy to see that $\mathcal C=\infty$ but $\mathcal A<\infty$. Thus our results extend and strengthen results in \cite{TXZ}. 
\vskip12pt
\begin{proof}[Proof of Theorem \ref{sec4theo1}.]
 If $f$ and $g$ are expressions, we will write $f\lesssim g$ if there is some constant $C$ such that $f\le C g$. To simplify the proof we denote ${\mathcal B}=\prod\limits_{i=1}^m\|b_i\|_{Lip^{\beta_i}}$, 
\[\widetilde{\psi}(t)=\left(\prod\limits_{i=1}^m|1-s_i(t)|^{\beta_i}\right)\psi(t),\]
and
\[
\overline{\psi}(t)=\left(\prod\limits_{i=1}^m|s_i(t)|^{-\frac{d+\gamma_i}{q_i}}\right)\widetilde{\psi}(t),
\] 

\vskip12pt
For any $x\in C_k$ and $b_i\in Lip^{\beta_i}$, it is easy to check that
\begin{align}\label{sec4eq6}
\prod\limits_{i=1}^m\left|b_i(x)-b_i(s_i(t)x)\right|
\leq & {\mathcal B}\cdot\left(\prod\limits_{i=1}^m\left|s_i(t)-1\right|^{\beta_i}\right)\cdot2^{k\left(\beta_1+\cdots+\beta_m\right)}.
\end{align}
Hence,
\begin{equation}
\begin{split}
&\qquad \|U^{m,n,\vec{b}}_{\psi,\vec{s}}(f_1, \ldots,f_m)\chi_k\|_{q,\omega}\\
&\leq \int_{[0,1]^n}\left(\int_{C_k}\prod\limits_{i=1}^m\left|f_i(s_i(t)x) \right|^q\prod\limits_{i=1}^m\left|b_i(x)-b_i(s_i(t)x)\right|^q\omega(x)dx\right)^{1/q}\psi(t)dt\\
&\lesssim2^{k(\beta_1+\cdots+\beta_m)}\,{\mathcal B} \int_{[0,1]^n}\left(\int_{C_k}\prod\limits_{i=1}^m\left|f_i(s_i(t)x) \right|^q\omega(x)dx\right)^{1/q}\widetilde{\psi}(t)dt\\
&\lesssim2^{k(\beta_1+\cdots+\beta_m)}\,{\mathcal B} \int_{[0,1]^n}\prod\limits_{i=1}^m\left(\int_{C_k}\left|f_i(s_i(t)x) \right|^{q_i}\omega_i(x)dx\right)^{1/{q_i}}\prod\limits_{i=1}^m\left(\int_{C_k}\omega_i(x)dx\right)^{1/{r_i}}\widetilde{\psi}(t)dt,\\
\end{split}
\notag
\end{equation}
where the last inequality is obtained by using H\"{o}lder's inequality. Noting that $\omega_i\in{\mathcal W}_{\gamma_i}$ gives 
\begin{equation}\label{sec4eq7}
\int_{C_k}\omega_i(x)dx\lesssim 2^{k(d+\gamma_i)}.
\end{equation}
Hence, by simple change of variables using (\ref{sec4eq5}), we get
\begin{align*}
&\qquad \|U^{m,n,\vec{b}}_{\psi,\vec{s}}(f_1, \ldots,f_m)\chi_k\|_{q,\omega}\\
&\;\lesssim2^{k(\beta_1+\cdots+\beta_m)}\,{\mathcal B} \int_{[0,1]^n}\prod\limits_{i=1}^m\left(\int_{s_i(t)C_k}\left|f_i(y) \right|^{q_i}|s_i(t)|^{-d-\gamma_i}\omega_i(y)dy\right)^{1/{q_i}}\prod\limits_{i=1}^m\left(\omega_i(C_k)\right)^{1/{r_i}}\widetilde{\psi}(t)dt\\
&\;\lesssim2^{k\left(\alpha-\alpha^\prime\right)}\,{\mathcal B}\int_{[0,1]^n}\prod\limits_{i=1}^m\left(\int_{s_i(t)C_k}\left|f_i(x) \right|^{q_i}\omega_i(x)dx\right)^{1/{q_i}}\overline{\psi}(t)dt.
\end{align*}
 Since $s_i(t)\ne 0$ almost everywhere $t\in [0,1]^n$, we could find an interger $\ell_i$ such that $2^{\ell_i-1}<s_i(t)\leq2^{\ell_i}$. Thus $s_i(t)x\in C_{k+\ell_i-1}\cup C_{k+\ell_i}$ for any $x\in C_k$. Therefore
\begin{equation}\label{sec4eq8}
\|U^{m,n,\vec{b}}_{\psi,\vec{s}}(f_1, \ldots,f_m)\chi_k\|_{q,\omega}
\lesssim2^{k\left(\alpha-\alpha^\prime\right)}\,{\mathcal B}\int_{[0,1]^n}\prod\limits_{i=1}^m\left(\|f_i\chi_{k+\ell_i-1}\|_{q_i,\omega_i}+\|f_i\chi_{k+\ell_i}\|_{q_i,\omega_i}\right)\overline{\psi}(t)dt.
\end{equation}
Now we consider the following cases:
\vskip12pt 
{\bf Case 1:} Suppose that $1\leq p<\infty$ and $\lambda\ge0$.  
By Lemma \ref{sec3lem1}, we have
\begin{align*}
&\|U^{m,n,\vec{b}}_{\psi,\vec{s}}(f_1, \ldots,f_m)\|_{M\dot{K}^{\alpha^\prime, \lambda}_{p, q}(\omega)}\\
=\;& \sup\limits_{k_0\in \mathbb Z}2^{-k_0\lambda}\left(\sum_{k=-\infty}^{k_0}2^{kp\alpha^\prime}\|U^{m,n,\vec{b}}_{\psi,\vec{s}}(f_1, \ldots,f_m)\chi_k\|_{q,\omega}^p\right)^{1/p}\\
\lesssim\;&{\mathcal B}\sup\limits_{k_0\in \mathbb Z}2^{-k_0\lambda} \left(\sum_{k=-\infty}^{k_0}2^{kp\alpha}\left(\int_{[0,1]^n}\prod\limits_{i=1}^m\left(\|f_i\chi_{k+\ell_i-1}\|_{q_i,\omega_i}+\|f_i\chi_{k+\ell_i}\|_{q_i,\omega_i}\right)\overline{\psi}(t)dt\right)^p\right)^{1/p}\\
\lesssim\;&{\mathcal B}\sup\limits_{k_0\in \mathbb Z}2^{-k_0\lambda} \int_{[0,1]^n} \left(\sum_{k=-\infty}^{k_0}2^{kp\alpha}\prod\limits_{i=1}^m\left(\|f_i\chi_{k+\ell_i-1}\|_{q_i,\omega_i}+\|f_i\chi_{k+\ell_i}\|_{q_i,\omega_i}\right)^p\right)^{1/p}\overline{\psi}(t)dt\\
\lesssim\;&{\mathcal B}\sup\limits_{k_0\in \mathbb Z}2^{-k_0\lambda} \int_{[0,1]^n} \left(\sum_{k=-\infty}^{k_0}2^{kp\alpha}\prod\limits_{i=1}^m\left(\|f_i\chi_{k+\ell_i-1}\|^p_{q_i,\omega_i}+\|f_i\chi_{k+\ell_i}\|^p_{q_i,\omega_i}\right)\right)^{1/p}\overline{\psi}(t)dt.\\
\end{align*}
It is worth noting that
\[
\prod\limits_{i=1}^m(a_{i,0}+b_{i,0})=\sum\limits_{j_1,\ldots,j_m=0,1}\prod\limits_{i=1}^ma_{i,j_i},
\]
where the sum is taken over all $(j_1,\ldots,j_m)$ for which $j_1,\ldots,j_m\in\{0,1\}$. Thus,
\begin{align*}
&\|U^{m,n,\vec{b}}_{\psi,\vec{s}}(f_1, \ldots,f_m)\|_{M\dot{K}^{\alpha^\prime, \lambda}_{p, q}(\omega)}\\
\lesssim\;&{\mathcal B}\sup\limits_{k_0\in \mathbb Z}2^{-k_0\lambda}\int_{[0,1]^n} \left(\sum_{k=-\infty}^{k_0}2^{kp\alpha}\left(\sum_{j_1,\ldots,j_m=0,1}\prod\limits_{i=1}^m\|f_i\chi_{k+\ell_i-j_i}\|^p_{q_i,\omega_i}\right)\right)^{1/p}\overline{\psi}(t)dt\\
\lesssim\;&{\mathcal B}\sup\limits_{k_0\in \mathbb Z}2^{-k_0\lambda}\int_{[0,1]^n} \left(\sum\limits_{j_1,\ldots,j_m=0,1}\left(\sum_{k=-\infty}^{k_0}2^{kp\alpha}\prod\limits_{i=1}^m\|f_i\chi_{k+\ell_i-j_i}\|^p_{q_i,\omega_i}\right)\right)^{1/p}\overline{\psi}(t)dt\\
\lesssim\;&{\mathcal B}\sup\limits_{k_0\in \mathbb Z}2^{-k_0\lambda}\int_{[0,1]^n}\sum\limits_{j_1,\ldots,j_m=0,1}\left(\sum_{k=-\infty}^{k_0}2^{kp\alpha}\prod\limits_{i=1}^m\|f_i\chi_{k+\ell_i-j_i}\|^p_{q_i,\omega_i}\right)^{1/p}\overline{\psi}(t)dt\\
\lesssim\;&{\mathcal B}\sup\limits_{k_0\in \mathbb Z}2^{-k_0\lambda}\int_{[0,1]^n}\sum_{j_1,\ldots,j_m=0,1}\left(\prod\limits_{i=1}^m\left(\sum_{k=-\infty}^{k_0}2^{kp_i\alpha_i}\|f_i\chi_{k+\ell_i-j_i}\|^{p_i}_{q_i,\omega_i}\right)\right)^{1/{p_i}}\overline{\psi}(t)dt\\
\lesssim\;&{\mathcal B}\int_{[0,1]^n}\sum_{j_1,\ldots,j_m=0,1}\left(\prod\limits_{i=1}^m\sup\limits_{k_0\in \mathbb Z}2^{-k_0\lambda_i}\left(\sum_{k=-\infty}^{k_0}2^{kp_i\alpha_i}\|f_i\chi_{k+\ell_i-j_i}\|^{p_i}_{q_i,\omega_i}\right)\right)^{1/{p_i}}\overline{\psi}(t)dt\\
\lesssim\;&{\mathcal B}\int_{[0,1]^n}\sum_{j_1,\ldots,j_m=0,1}\left(\prod\limits_{i=1}^m\|f_i\|_{M\dot{K}^{\alpha_i, \lambda_i}_{p_i, q_i}(\omega_i)}\left(\prod\limits_{i=1}^m 2^{(\ell_i-j_i)\lambda_i}\right)\left(\prod\limits_{i=1}^m2^{(-\ell_i+j_i)\alpha_i}\right)\right)\overline{\psi}(t)dt\\
\lesssim\;&{\mathcal B}\int_{[0,1]^n}\sum_{j_1,\ldots,j_m=0,1}\left(\prod\limits_{i=1}^m\|f_i\|_{M\dot{K}^{\alpha_i, \lambda_i}_{p_i, q_i}(\omega_i)}\right)\left(\prod\limits_{i=1}^m|s_i(t)|^{\lambda_i-\alpha_i}\right)\overline{\psi}(t)dt\\
\lesssim\;&{\mathcal B}\left(\prod\limits_{i=1}^m\|f_i\|_{M\dot{K}^{\alpha_i, \lambda_i}_{p_i, q_i}(\omega_i)}\right)\int_{[0,1]^n}\sum_{j_1,\ldots,j_m=0,1}\left(\prod\limits_{i=1}^m|s_i(t)|^{\lambda_i-\alpha_i}\right)\overline{\psi}(t)dt\\
\lesssim\;&{\mathcal B}\left(\prod\limits_{i=1}^m\|f_i\|_{M\dot{K}^{\alpha_i, \lambda_i}_{p_i, q_i}(\omega_i)}\right)\int_{[0,1]^n}\left(\prod\limits_{i=1}^m|s_i(t)|^{\lambda_i-\alpha_i}\right)\overline{\psi}(t)dt.\\
\end{align*}
From the hypothesis (\ref{sec4eq3}), we deduce that $\int\limits_{[0,1]^n}\left(\prod\limits_{i=1}^m|s_i(t)|^{\lambda_i-\alpha_i}\right)\overline{\psi}(t)dt$ is finite. Thus $U^{m,n,\vec{b}}_{\psi,\vec{s}}$ is bounded from $M\dot{K}^{\alpha_1, \lambda_1}_{p_1, q_1}(\omega_1)\times \cdots \times M\dot{K}^{\alpha_m, \lambda_m}_{p_m, q_m}(\omega_m)$ to $M\dot{K}^{\alpha^\prime, \lambda}_{p, q}(\omega)$. 
\vskip12pt
{\bf Case 2:} Suppose that $0<p<1$ and $\lambda>0$. The following lemma will be useful for proving the second case. By (\ref{sec4eq6}) and Lemma \ref{sec3lem2}, we have 
\begin{align*}
&\|U^{m,n,\vec{b}}_{\psi,\vec{s}}(f_1, \ldots,f_m)\|_{M\dot{K}^{\alpha^\prime, \lambda}_{p, q}(\omega)}\\
\lesssim\;&{\mathcal B}\sup\limits_{k_0\in \mathbb Z}2^{-k_0\lambda} \left(\sum_{k=-\infty}^{k_0}2^{kp\alpha}\left(\int_{[0,1]^n}\prod\limits_{i=1}^m\left(\|f_i\chi_{k+\ell_i-1}\|_{q_i,\omega_i}+\|f_i\chi_{k+\ell_i}\|_{q_i,\omega_i}\right)\overline{\psi}(t)dt\right)^p\right)^{1/p}\\
\lesssim\;&{\mathcal B}\sup\limits_{k_0\in \mathbb Z}2^{-k_0\lambda}\left(\sum_{k=-\infty}^{k_0}2^{kp\alpha}\left(\int_{[0,1]^n}\prod\limits_{i=1}^m\left(2^{(k+\ell_i-1)(\lambda_i-\alpha_i)}+2^{(k+\ell_i)(\lambda_i-\alpha_i)}\right)\|f_i\|_{M\dot{K}^{\alpha_i,\lambda_i}_{p_i,q_i}(\omega_i)}\overline{\psi}(t)dt\right)^p\right)^{1/p}\\
\lesssim\;&{\mathcal B}\left(\prod\limits_{i=1}^m\|f_i\|_{M\dot{K}^{\alpha_i, \lambda_i}_{p_i, q_i}(\omega_i)}\right)\cdot\sup_{k_0\in\mathbb Z}2^{-k_0\lambda}\left(\sum\limits_{k=-\infty}^{k_0}2^{kp\alpha}\cdot\prod\limits_{i=1}^m2^{k(\lambda_i-\alpha_i)p}\left(\int\limits_{[0,1]^n}\prod\limits_{i=1}^m|s_i(t)|^{\lambda_i-\alpha_i}\overline{\psi}(t)dt\right)^p\right)^{1/p}\\
\lesssim\;&{\mathcal B}\left(\prod\limits_{i=1}^m\|f_i\|_{M\dot{K}^{\alpha_i, \lambda_i}_{p_i, q_i}(\omega_i)}\right)\left(\int\limits_{[0,1]^n}\prod\limits_{i=1}^m|s_i(t)|^{\lambda_i-\alpha_i}\overline{\psi}(t)dt\right)\sup_{k_0\in\mathbb Z}\left(\sum\limits_{k=-\infty}^{k_0}2^{(k-k_0)\lambda p}\right)^{1/p}.
\end{align*}
We here remind that $\lambda=\lambda_1+\cdots+\lambda_m$ and $\alpha=\alpha_1+\cdots+\alpha_m$.  Since $\lambda>0$, the series $\sum\limits_{k=-\infty}^{k_0}2^{(k-k_0)\lambda p}$ is convergent and its sum is a constant, not depend on $k_0$. Therefore
\[
\|U^{m,n,\vec{b}}_{\psi,\vec{s}}(f_1, \ldots,f_m)\|_{M\dot{K}^{\alpha^\prime, \lambda}_{p, q}(\omega)}\lesssim{\mathcal B}\left(\prod\limits_{i=1}^m\|f_i\|_{M\dot{K}^{\alpha_i, \lambda_i}_{p_i, q_i}(\omega_i)}\right)\left(\int\limits_{[0,1]^n}\prod\limits_{i=1}^m|s_i(t)|^{\lambda_i-\alpha_i}\overline{\psi}(t)dt\right).
\]
This completes our proof of Theorem \ref{sec4theo1}.
\end{proof}

\vskip12pt\noindent
{\bf Acknowledgement.} This paper is supported by Vietnam National Foundation for Science and Technology Development (NAFOSTED) under grant number 101.02 - 2014.51.



\begin{thebibliography}{MTW1} 
\bibitem{CH}  N.M. Chuong, H.D. Hung, Weighted $L^p$ and weighted $BMO-$bounds for a new generalized  weighted Hardy-Ces\`{a}ro operator. Integral Transforms Spec. Funct. {\bf 25} (2014), no. 9, 697--710.
\bibitem{CHD}  N.M. Chuong, H.D. Hung, D.V. Duong, Bounds for the weighted Hardy-Ces\`{a}ro operator and its commutator on weighted Morrey-Herz type spaces. Submitted.
\bibitem{CRW} R.R. Coifman, R. Rochberg  and G. Weiss, Factorization theorems for Hardy spaces in several variables.  Ann. of Math. (2) {\bf 103} (1976), no. 3, 611--635. 
\bibitem {FGL} Z.W. Fu, S. L. Gong, S. Z. Lu and W. Yuan, Weighted multilinear Hardy operators and commutators. Forum Math. {\bf27} (2015), no. 5, 2825--2851.
\bibitem {FLL} Z.W. Fu, Z.G. Liu, S.Z. Lu, Commutators of weighted Hardy operators on $\mathbb R^n$. Proc. Amer. Math. Soc. {\bf 137} (2009), no. 10, 3319--3328.
\bibitem {FL} Z.W. Fu, S.Z. Lu, A remark on weighted Hardy-Littlewood averages on Herz spaces. Advances in Math. {\bf37} (2008), 632--636.
\bibitem {FL2} Z.W. Fu, S.Z. Lu, Commutators of generalized Hardy operators. Math. Nachr. {\bf282} (2009), no. 6, 832--845.
\bibitem {FL1} Z.W. Fu, S.Z. Lu, Weighted Hardy operators and commutators on Morrey spaces. Front. Math. China {\bf5} (2010), no. 3, 531--549.
\bibitem{GFM} S. Gong, Z.W. Fu, B. Ma, Weighted multilinear Hardy  operators on Herz-type spaces. The Sci. World. J., {\bf 2014} (2014), 420480, 10 pages.
\bibitem{GY}  G. Gao and Y. Zhong, Some inequalities for Hausdorff operators. Math. Inequal. Appl. {bf17} (2014), no. 3, 1061--1078.
\bibitem{HK} H.D. Hung, L.D. Ky, New weighted multilinear operators and commutators of Hardy-Ces\`{a}ro type. Acta Math. Sci., Ser. B, Engl. Ed. {\bf 35B} (2015), no. 6, 1-15.
\bibitem {KS} Y. Komori, S. Shirai, Weighted Morrey spaces and a singular integral operator. Math. Nachr. {\bf282} (2009), no. 2, 219--231 (2009).
\bibitem{LF} Z.G. Liu, Z.W. Fu, Weighted Hardy-Littlewood averages on Herz spaces. Acta Mathematics Sinica. {\bf49}, 1085--1090 (2006).

\bibitem{Xiao}  J. Xiao, $L^p$ and BMO bounds of weighted Hardy-Littlewood Averages. J. Math. Anal. Appl. {\bf 262} (2001), 660-666.
\bibitem{YSY}  W. Yuan, W. Sickel, D. Yang, Morrey and Campanato meet Besov, Lizorkin and Triebel. Lecture Notes in Math. 2005 (Springer, Berlin, Germany, 2010).
\bibitem{TXZ} C. Tang, F. Xue, Y. Zhou, Commutators of weighted Hardy operators on Herz-type spaces. Ann. Pol. Math. {\bf101} (2011), no. 3, 267 --273.
\end{thebibliography}
\end{document}